\documentclass[11pt,english]{article}
\usepackage[T1]{fontenc}
\usepackage[latin9]{inputenc}
\usepackage{wrapfig}
\usepackage{textcomp}
\usepackage{amsmath}
\usepackage{amsthm}
\usepackage{amssymb}
\usepackage{graphicx}

\makeatletter
\theoremstyle{plain}
\newtheorem{thm}{\protect\theoremname}
  \theoremstyle{plain}
  \newtheorem{prop}[thm]{\protect\propositionname}
  \theoremstyle{remark}
  \newtheorem*{rem*}{\protect\remarkname}

\usepackage{mathrsfs}\usepackage{amsthm}\usepackage{amscd}

\usepackage{hyperref}
\usepackage{graphics}
\usepackage{a4wide}
\@ifundefined{definecolor}
 {\usepackage{color}}{}

\newcounter{EQNR}
\renewcommand{\contentsname}{Contents\\
{\footnotesize\normalfont(A table of contents should normally not
be included)}}

\makeatother

\usepackage{babel}
  \providecommand{\propositionname}{Proposition}
  \providecommand{\remarkname}{Remark}
\providecommand{\theoremname}{Theorem}

\begin{document}

\title{Generalized Lyapunov exponents and aspects of the theory of deep
learning}

\author{Anders Karlsson\footnote{The author was supported in part by the Swedish Research Council grant 104651320 and the Swiss NSF grants 200020-200400 and 200021-212864.}}

\date{December 26, 2022}
\maketitle
\begin{abstract}
We discuss certain recent metric space methods and some of the possibilities
these methods provide, with special focus on various generalizations
of Lyapunov exponents originally appearing in the theory of dynamical
systems and differential equations. These generalizations appear for
example in topology, group theory, probability theory, operator theory
and deep learning. 
\end{abstract}

\section{Introduction}

The law of large numbers states that for a sequence of independent,
identically distributed (i.i.d.) random variables $X_{1},X_{2},...,X_{n}$
with finite expectation, 
\begin{equation}
\left(X_{1}+X_{2}+...+X_{n}\right)/n\rightarrow\mathbb{E}\left[X_{1}\right]\label{eq:LLN}
\end{equation}
almost surely as $n\rightarrow\infty$. Bellman \cite{Be54}, Furstenberg
\cite{Fu63} and others asked whether in some situations there could
exist a similar limit law for products 

\[
u(n):=g_{1}g_{2}g_{3}...g_{n}
\]
of i.i.d. noncommutative operations $g_{1,}g_{2},...,g_{n}$. Such
products appear for example as solutions to difference equations with
random coefficients, or from time-one maps of the solutions of continuous
models, say from a stochastic PDE. In addition to mathematics and
physics, one can find papers in biology, epidemiology, medicine, and
economics leading to random products of noncommuting transformations
\cite{BHS21,CDS09,IS96,Neu19}. Compositional products is also one
of the key features of deep learning as will be highlighted below.

Note that in contrast to \eqref{eq:LLN} it is unclear how to form
an average in the noncommutative setting. Important partial answers
to the above question was obtained at the end of the 1960s (\cite{Ki68,O68})
and later as one aspect of random walks on groups, see for example
\cite{Gu80,KM99,Ka00,Er10,BQ16,MT18,Zh22}. A quite general affirmative
answer to the question of a limit law for noncommuting random products
was provided in \cite{KL06,GK20}, see Theorem \ref{thm:ergodic}
below.

In ergodic theory one formalizes the setting as follows, more general
than the i.i.d. assumption. Let $(\Omega,\mu)$ be a measure space
with $\mu(\Omega)=1$. Let $T:\Omega\rightarrow\Omega$ be a measurable
map preserving the measure. We furthermore assume ergodicity, which
is an irreducibility assumption that states that up to measure zero
there are no $T$-invariant subsets of $\Omega$. Given a measurable
map $g:\Omega\rightarrow G$ (assigning some measurable structure
on the group; $g$ is what a probabilist would call a random variable),
we define the following \emph{ergodic cocycle}:
\[
u(n,\omega):=g(\omega)g(T\omega)...g(T^{n-1}\omega).
\]

In addition, one needs to assume that the cocycle is integrable which
means that the integral over $\Omega$ of the ``size'' of $g(\omega)$
is finite. 

For matrices, a first answer was provided by Furstenberg-Kesten for
the norm of the matrices, and a more precise answer was given later
in the 1960s by Oseledets in his multiplicative ergodic theorem. One
can view this as a random spectral theorem, intuitively it says that
the random product behaves in the same way as the powers of one single
``average'' matrix:
\begin{thm}
\emph{(Oseledets' multiplicative ergodic theorem \cite{O68})} \label{thm:Osel}Given
an integrable ergodic cocycle $A(n,\omega)=g_{n}g_{n-1}...g_{1}$
of invertible matrices, there are a.s. a random filtration of subspaces
$0=V_{0}\subset V_{1}\subset...\subset V_{k}=\mathbb{R}^{d}$ and
numbers $\lambda_{1}<\lambda_{2}<...<\lambda_{k}$ such that 
\[
\lim_{n\rightarrow\infty}\frac{1}{n}\log\left\Vert A(n,\omega)v\right\Vert =\lambda_{i}
\]
whenever $v\in V_{i}\setminus V_{i-1}.$
\end{thm}

The case $d=1$ is the Birkhoff ergodic theorem generalizing (\ref{eq:LLN}).
The numbers $\lambda_{i}$ are called \emph{Lyapunov exponents.} This
is a fundamental theorem in the theory of differentiable dynamical
systems and has also many other applications, the most spectacular
such is Margulis' proof \cite{M75} of his super-rigidity theorem.
As for physics, Nobel laureate Parisi wrote in the foreword of \cite{CPV93}
that ``The properties of random matrices and their products form
a basic tool, whose importance cannot be underestimated. They play
a role as important as Fourier transforms for differential equations.''

Now compare the above with the following. Let $\Sigma$ be an oriented
closed surface of genus $g\geq2$. Let $\mathcal{S}$ denote the isotopy
classes of simple closed curves on $M$ not isotopically trivial.
For a Riemannian metric $\rho$ on $\Sigma$, let $l_{\rho}(\beta)$
be the infimum of the length of curves isotopic to $\beta.$ In a
legendary preprint from 1976 \cite{T88}, Thurston announced the following
(the details are worked out in \cite[Théorème Spectrale]{FLP79} using
foliation theory): 
\begin{thm}
\label{thm:Thurston-1}\emph{(Thurston's spectral theorem for surface
diffeomorphisms \cite{T88}) }Given a diffeomorphism $f$ of a surface
$\Sigma$ of genus $g\geq2$. Then there is filtration of subsurfaces
$Y_{1}\subset Y_{2}\subset...\subset Y_{k}=\Sigma$ and algebraic
integers $\lambda_{1}<\lambda_{2}<..<\lambda_{k}$ such that
\[
\lim_{n\rightarrow\infty}\frac{1}{n}\log l_{\rho}(f^{n}c)=\lambda_{i}
\]
whenever the simple closed curve $c$ can be isotoped to a curve contained
in $Y_{i}$ but not in $Y_{i-1}$. 
\end{thm}

\begin{wrapfigure}{o}{0.5\columnwidth}%
\includegraphics[scale=0.75]{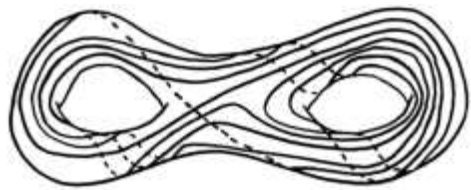}\caption{$f^{n}(c)$ (illustration from \cite{T88})}
\end{wrapfigure}%
This is analogous to a simple statement for linear transformations
$A$ in finite dimensions (which corresponds to the Oseledets theorem
in the case $A(n,\omega)=A^{n}$): given a vector $v$ there is an
associated exponent $\lambda$ (absolute value of an eigenvalue),
such that
\[
\lim_{n\rightarrow\infty}\left\Vert A^{n}v\right\Vert ^{1/n}=\lambda.
\]

To spell out the analogy: a diffeomorphism $f$ instead of a linear
transformation $A$, a length instead of a norm, and a curve $\alpha$
instead of a vector $v$. And the answer is given in similar terms:
Lyapunov exponents and associated filtration of subspaces and subsurfaces
respectively.

A \emph{weak metric space }(following the terminology of for example
\cite{GuW22}) is a set $X$ equipped with a function $X\times X\rightarrow\mathbb{R}$
such that 
\[
d(x,x)=0
\]
and
\[
d(x,y)\leq d(x,z)+d(z,y)
\]
for all points $x,y,z\in X$. A map $f:X\rightarrow X$ is called
\emph{nonexpansive} if 
\[
d(f(x),f(y))\leq d(x,y)
\]
for all $x,y\in X$.

The following multiplicative ergodic theorem was proved for isometries
in \cite{KL06} and in general in \cite{GK20}.
\begin{thm}
\label{thm:ergodic}\emph{(Ergodic theorem for noncommuting random
products, \cite{KL06,GK20}) }Let $u(n,\omega)$ be an integrable
ergodic cocycle of nonexpansive maps of a weak metric space $(X,d)$
and such that $\omega\mapsto u(n,\omega)x$ is measurable. Then there
exists a.s. a metric functional $h$ of $X$ such that
\[
\lim_{n\rightarrow\infty}-\frac{1}{n}h(u(n,\omega)x)=\lim_{n\rightarrow\infty}\frac{1}{n}d(x,u(n,\omega)x).
\]
\end{thm}

\begin{figure}
\includegraphics[scale=0.4]{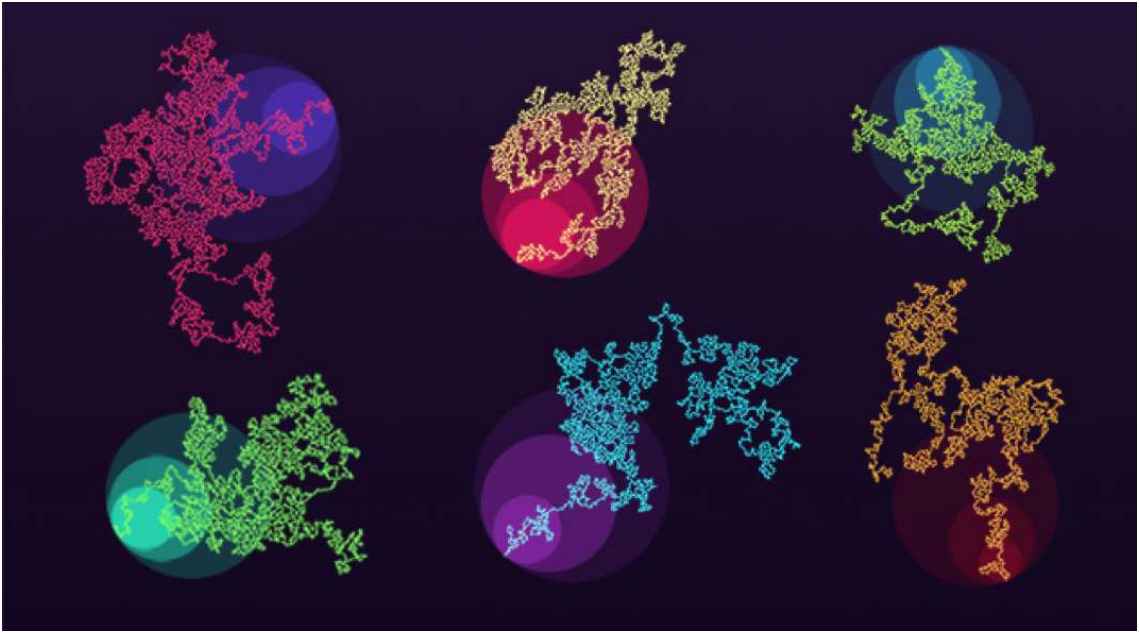}\caption{Illustration of the convergence in direction of $u(n,\omega)x$ as
$n\rightarrow\infty$ in six experiments in a hyperbolic disk. The
colored disks are horodisks that the random walks trajectories go
deeper and deeper into with time. (Image by Cécile Bucher.)}
\end{figure}
Metric functionals are almost what is usually called horofunctions,
see section \ref{sec:Metric-preliminaries}. This theorem when specialized
to $X$ a symmetric space of nonpositive curvature, Gromov hyperbolic
space, or CAT(0) space, recovers some previous results mentioned above
by Oseledets, Furstenberg, Kaimanovich, and Karlsson-Margulis. It
also implies random mean ergodic theorems of Ulam-von Neumann, Kakutani,
and Beck-Schwartz, see \cite{GK20}, and it holds even when the traditional
mean ergodic theorem fails \cite{K21b}. Theorem \ref{thm:ergodic}
furthermore provides generalized laws of large numbers with convex
moments \cite{KMo08}, and has found application to random walks on
groups and bounded harmonic functions on manifolds \cite{KL07,KL07b}
without knowing anything specific about the metric functionals in
these settings. I want to emphasize that Theorem \ref{thm:ergodic}
applies in particular to every random walks with finite first moment
on \emph{any} finitely generated group.

Moreover, using Theorem \ref{thm:ergodic}, Horbez, building on the
approach of \cite{K14}, could establish the following random extension
of Thurston's theorem: 
\begin{thm}
\emph{(\label{thm:(Random-spectral-theorem}Random spectral theorem
of surface homeomorphisms \cite{K14,H16})} Let $v(n,\omega)=A(t^{n-1}\omega)...A(T\omega)A(\omega)$
be an integrable i.i.d random product of homeomorphisms of a closed
surface $\Sigma$ of genus $g\geq2$. Then there is a (random) filtration
of subsurfaces $Y_{1}\subset Y_{2}\subset...\subset Y_{k}=\Sigma$
and (deterministic) exponents $\lambda_{1}<\lambda_{2}<..<\lambda_{k}$
such that
\[
\lim_{n\rightarrow\infty}\frac{1}{n}\log l_{\rho}(v(n,\omega)c)=\lambda_{i}
\]
whenever the simple closed curve $c$ can be isotoped to a curve contained
in $Y_{i}$ but not in $Y_{i-1}$. Here $l_{\rho}$ is the minimal
length in the isotopy class in some fixed Riemannian metric $\rho.$ 
\end{thm}

The exponents $\lambda_{i}$ are a type of generalized Lyapunov exponents
that perhaps could be called \emph{topological Lyapunov exponents
for surface homeomorphisms.} A different approach was provided in
\cite{K18} which showed how to get the top exponent for ergodic cocylces
of homeomorphisms, using the metric ideas and a lemma in \cite{KM99}
combining it with results in \cite{LRT12}. In both cases, the proofs
use Thurston's asymmetric metric. A prior study of random walks on
the mapping class groups was carried out in \cite{KM96}, which showed
in particular that under a non-elementary assumption the random walk
converges to uniquely ergodic foliations, in which case it follows
from \cite{K14} that there is only one exponent.

Actually the main results of Horbez' paper \cite{H16} concern instead
random walks on the outer automorpshism group of free groups, giving
a result very similar to Theorem \ref{thm:(Random-spectral-theorem}.
In order not to have to explain notations from the important subject
of automorphisms of free groups, I will not state it here. The proof
goes via a determination of the metric functionals of the outer space
and an application of Theorem \ref{thm:ergodic}. To get all generalized
Lyapunov exponents Horbez then studies the set of stationary measures
on the boundary of outer space in parallel to works in the matrix
case of Furstenberg-Kifer and Hennion.

We thus see three settings, linear transformations in finite dimensions,
surface homeomorphisms, and automorphisms of free groups, that are
not merely analogous but the corresponding ``law of large numbers''
can be deduced ultimately from the same theorem. The strategy is:
\begin{itemize}
\item Instead of looking at the underlying space where the linear maps,
homeomorphisms, group automorphisms etc act, we lift the action to
a more abstract space, a\emph{ moduli space} as it were, of positive
structures on the corresponding underlying space. 
\item On that associated space there is often an invariant metric.
\item Employ the noncommutative ergodic theorem in terms of metric functionals
and interpret the result as concretely as possible.
\end{itemize}
The very last part can in fact often be done, as testified by Theorems
\ref{thm:Osel} and \ref{thm:(Random-spectral-theorem} above which
include no reference to metric functionals (or horofunctions), and
likewise in the last section on deep learning. Sometimes, like in
complex analysis, Cayley graphs, or maps of cones there is no need
to pass to an auxiliary space in order to find an invariant metric
for the transformations in question.

Further generalized Lyapunov exponents could perhaps also be defined
for higher dimensional diffeomorphisms, using their isometric action
on Ebin's space of Riemannian metrics on a fixed compact manifold
\cite{Eb68} and an investigation of the metric functionals. Some
progress and possibilities are pointed out in section \ref{sec:Diffeomorphisms}.

A different direction, using some of the arguments in \cite{KL06},
was developed by Masai \cite{Ma21}, namely for surface bundles over
a circle, he showed that for pseudo-Anosov maps the translation length
(=``top Lyapunov exponent'' in the terminology of the present paper)
in a certain weak metric equals the 3-dimensional hyperbolic volume
of the mapping torus that the map defines.

\paragraph*{Ackonwledgement: }

This text was written in connection with the conference ``New Trends
in Lyapunov Exponents'' in Lisbon 2022. I thank the organizers and
especially Pedro Duarte for the invitation to this very pleasant and
stimulating week. I also thank Alex Blumenthal for helpful discussions
related to the topics of this paper during this meeting. 

\section{Deep learning}

Deep learning\textbf{ }provided Artificial Intelligence (AI) with
a long sought-after new tool that moreover exceeded all expectation,
as was realized starting from around 2012. The development of deep
neural network had begun much earlier and Bengio, Hinton, and LeCun
received the 2018 Turing award for these methods. Part of the prize
citation stated ``By dramatically improving the ability of computers
to make sense of the world, deep neural networks are changing not
just the field of computing, but nearly every field of science and
human endeavor.'' 

The remarkable success of these methods indicates that real-life data
tend to have a compositional structure. More precisely, given a learning
task, one seeks maps $g_{1},g_{2}...g_{n}$ such that their composition
\[
u(n):=g_{1}g_{2}...g_{n}
\]
applied to the input data should be close to the desired output (possibly
after applying a certain \emph{decision function }$f$). The \emph{depth}
$n$ can be several hundred. The maps are often of the form $g_{i}(x)=\sigma(W_{i}x+b_{i})$
where $\sigma$ is a fixed nonlinear function, called \emph{activation
function}, applied componentwise, $W_{i}$ is a $d\times d$ matrix,
called \emph{weights}, and $b_{i}$ is a vector in $\mathbb{R}^{d}$,
called \emph{bias vector}. The dimension $d$ is called the \emph{width.}
The nonlinearity, inspired by our brains (\cite{MP43,Ro58}), is crucial
(for one thing, the composition of affine maps are again affine, and
likewise the composition of polynomials is again a polynomial). Some
standard choices are, sigmoid/logistic function ($1/(1+e^{-t})$),
TanH ($\tanh(t)$), and ReLU (Rectified Linear Unit, $\sigma(t)=\max\{0,t\}$),
with different features and advantages. ReLU has been observed to
work particularly well, generally better than smooth functions. 

In practice, the weights and biases in the neural network are first
randomly selected (\emph{initialization}) and then optimized by stochastic
gradient descent on a chosen loss function specific to the task (\emph{training}).

\begin{figure}
\includegraphics[scale=0.5]{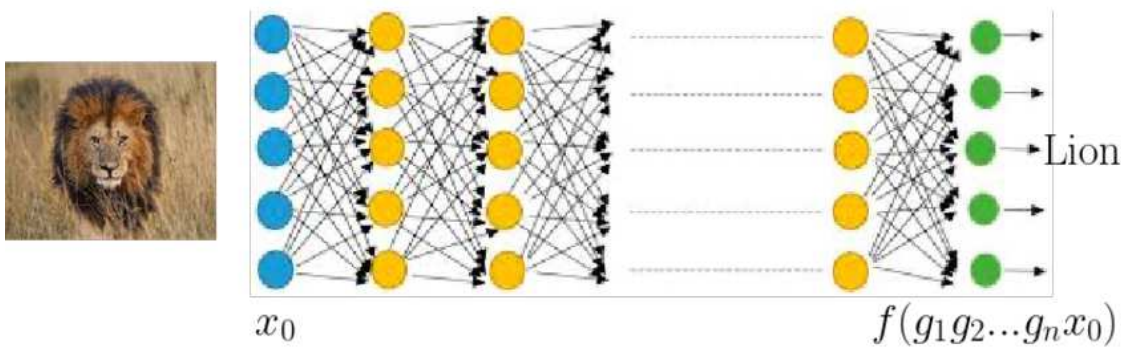}\caption{A deep neural network}
\end{figure}

Much of the subject of deep learning consists of empirical observations,
there is no or little theoretical understanding, as remarked by many
authors. According to \cite{Se20} ``A mathematical theory of deep
learning would illuminate how they function, allow us to assess the
strengths and weaknesses of different network architectures, and lead
to major improvements.'' Which type of layer maps to take, which
$\sigma$, how many layers $n$, how many nodes $d$ in each layer,
how to best find the parameters, how stable the solution is under
random perturbation (such as the \emph{drop-out procedure}) are some
of the questions of important practical concern. The need for a theoretical
understanding, instead of relying on black-box techniques, is also
expressed by practitioners, this lack of theory hinders their work.

One of the remarkable features that is not understood, is why deep
neural networks generally seem to mostly avoid the problem of overfitting
which is a phenomenon in traditional statistics. The latter typically
happens when approximating some data with a polynomial of very high
degree, the curves go through all the sample or training data, but
inbetween these points of perfect fit it can fluctuate wildly, related
to the Runge phenomenon. This is clearly undesirable. 

There are several ways random products $u(n):=g_{1}g_{2}...g_{n}$
of noncommuting nonlinear maps appear in deep learning:
\begin{enumerate}
\item \emph{Random initialization} see \cite{NBYS22} for a review
\item \emph{Drop-out regularization }which in particular is used to verify
robustness of the obtained error minimizer, and also a way training
the network \cite{SHK14}
\item \emph{Bayesian learning }\cite{Ne12}
\item Learning that combines taking some maps $g_{i}$ at random and optimize
the remaining one, a procedure with apparently good performance that
speeds up the training significantly \cite{BT22}
\end{enumerate}
The first two concepts are so fundamental in the current state-of-the-art
that one encounters them after any couple of first lectures on deep
learning. Hanin wrote in \cite{Ha21} ``Beyond illuminating the properties
of networks at the start of training, the analysis of random neural
networks can reveal a great deal about networks after training as
well.''

Moreover, as Avelin pointed out to me, not only the initialization
but also the training (stochastic gradient descent) actually involves
a random product of transformations. Thus we see compositional product
of random operations appearing in several ways in deep learning. Since
the number of layer maps can approach a thousand, it should make limit
theorems as discussed in this article very relevant, as will be discussed
in the last section.

\section{Elements of a metric functional analysis\label{sec:Metric-preliminaries}}

A metric space is a set equipped with a distance function $d(x,y)$
that is semi-positive, symmetric and satisfy the triangle inequality.
The author argued in \cite{K21b} that it is useful to develop parts
of metric geometry in analogy with linear functional analysis. 

Sometimes various ways of weakening the notion of a metric are useful:
pseudo-metrics arise naturally in complex analysis and here we will
also allow for asymmetric metrics (as in Thurston's metric). Moreover
we will let $d$ possibly to take negative values, useful for topical
maps. In other words, we consider weak metrics as defined in the introduction.
Weak metrics (but taking only nonnegative values) were in fact already
of interests to people like Heinz Hopf in the 1940s, see \cite{Ri43}
pointed out in \cite{PT09}.

Note that a symmetrization such as $D(x,y):=\max\left\{ d(x,y),d(y,x)\right\} $
(or the sum) is nonnegative following from
\[
0=d(x,x)\leq d(x,y)+d(y,x).
\]
 With the pseudo-metric $D$ one defines a topology on $X$. In case
the separation axiom holds, i.e. that $d(x,y)=0$ implies $x=y$,
then $D$ is a genuine metric.

A map $f:X\rightarrow Y$ is \emph{nonexpansive }if
\[
d_{Y}(f(x_{1}),f(x_{2}))\leq d_{X}(x_{1},x_{2})
\]
for all $x_{1},x_{2}\in X$. Note that compositions of nonexpansive
maps remain nonexpansive. Moreover, note that if we pass to a symmetrization
of a weak metric, then $f$ remains nonexpansive.

Let $(X,d)$ be a weak metric space. We will now define the \emph{metric
compactification }of $X$, which will provide a weak topology with
compactness properties in the metric setting. (I learnt from Cormac
Walsh that this construction works without essential changes to asymmetric
metrics, see \cite{AGW09} and \cite{W14} which inspired \cite{K14},
and see also the more recent paper \cite{GuW22}.) 

Let $F(X,\mathbb{R})$ be the space of continuous functions $X\rightarrow\mathbb{R}$
equipped with the topology of pointwise convergence. Given a base
point $x_{0}$ of the metric space $X$, let 
\[
\Phi:X\rightarrow\mathrm{F}(X,\mathbb{R})
\]
be defined via
\[
x\mapsto h_{x}(\cdot):=d(\cdot,x)-d(x_{0},x).
\]

\begin{prop}
The map $\Phi$ is a well-defined continuous map, and it is injective
if $d$ separates points. The closure $\overline{\Phi(X)}$ is compact.
\end{prop}

\begin{proof}
By the triangle inequality
\[
h_{x}(y)-h_{x}(z)=d(y,x)-d(z,x)\leq d(y,z)
\]
\[
-h_{x}(y)+h_{x}(z)=-d(y,x)+d(z,x)\leq d(z,y)
\]

therefore
\[
\left|h_{x}(y)-h_{x}(z)\right|\leq\max\left\{ d(y,z),d(z,y)\right\} ,
\]
which in particular implies that $h_{x}$ is continuous. This inequality
clearly passes to the closure. The map $\Phi$ is continuous since
\[
\left|h_{x}(z)-h_{y}(z)\right|=\left|d(z,x)-d(x_{0},x)-d(z,y)+d(x_{0},y)\right|\leq
\]
\[
2\max\left\{ d(x,y),d(y,x)\right\} 
\]
by the usual triangle inequality and the one in $X$.

Suppose that $d$ separates points, then given two points $x$ and
$y$, assume that $d(x_{0},x)\geq d(x_{0},y)$. If $d(x,y)>0$, then
\[
h_{x}(x)-h_{y}(x)=-d(x_{0},x)-d(x,y)+d(x_{0},y)\leq-d(x,y)<0
\]
shows that the two functions are different. In case $d(y,x)>0$, then
\[
h_{y}(x)-h_{x}(x)=d(y,x)-d(x_{0},y)+d(x_{0},x)\geq d(y,x)>0.
\]
These two cases cover all possibilities in view of the remark above
about $D$, and proves the injectivity.

Finally note that by the triangle inequality
\[
-d(x_{0},y)\leq h_{x}(y)\leq d(y,x_{0}).
\]
In view of this and the topology of pointwise topology which is the
product topology, the Tychonov theorem implies that $\overline{\Phi(X)}$
is compact.
\end{proof}
This proposition is the metric space analog of the Banach\textendash Alaoglu\emph{
}theorem. We call $\overline{X}:=\overline{\Phi(X)}$ the \emph{metric
compactification of }$X$ and its elements \emph{metric functionals,
}which recently has been described concretely in a variety of metric
spaces. This development is in parallel to the determination of dual
spaces in the beginning of functional analysis a century ago. I reserve
the more commonly used word \emph{horofunction} for limits in the
topology of uniform convergence on bounded sets (Gromov's choice of
topology in \cite{Gr81} considering genuine metric spaces) of $h_{x_{n}}$
for sequences $x_{n}$ such that $d(x_{0},x_{n})\rightarrow\infty$.

\textbf{A note on the proof of Theorem \ref{thm:ergodic} in the weak
metric case:} while \cite{KL06} considered a skew-product extension
to the boundary, the \cite{GK20} paper established a new substantial
refinement of the subadditive ergodic theorem \cite{Ki68}, which
we feel is a refinement of the fundamental theorem of Kingman that
has potential for further use. It has indeed already found independent
dynamical applications in \cite{KS19,CD20,ZC21}. In his book \cite{Sz01}
computer scientist Szpankowski explains why subadditivity and the
subadditive ergodic theorem are fundamental for the analysis of algorithms.

As observed in \cite{K14} the noncommutative ergodic theorem works
with an asymmetric metric $d$. Here is the verification that it works
even for weak metrics. First, it is of importance that 
\[
a(n,\omega):=d(x_{0},u(n,\omega)x_{0})
\]
is a subadditive cocycle. This is verified as follows: 
\[
d(x_{0},u(n+m,\omega)x_{0})\leq d(x_{0},u(n,\omega)x_{0})+d(u(n,\omega)x_{0},u(n+m,\omega)x_{0})
\]
\[
\leq d(x_{0},u(n,\omega)x_{0})+d(x_{0},u(m,T^{n}\omega)x_{0}).
\]
Kingman's subadditive ergodic theorem then asserts that 
\[
\lim_{n\rightarrow\infty}\frac{1}{n}d(x,u(n,\omega)x)
\]
exists a.e. under the integrability condition
\[
\int_{\Omega}\left|d(x_{0},u(1,\omega)x_{0})\right|d\mu<\infty.
\]
Note that nothing here depends on the choice of $x_{0}$ including
the value of the limit therefore written with a general point $x$.
And if we assume that $(\Omega,\mu,T)$ is an ergodic system then
this ``top Lyapunov'' exponent is deterministic, i.e. essentially
constant in $\omega$.

The proof of Theorem \ref{thm:ergodic} in the weak metric setting
now follows exactly \cite[section 3]{GK20} except that in that reference
the order in which the metric is written is reversed. 

\textbf{A possible future direction: }It seems plausible that often
the limits 
\[
\lim_{n\rightarrow\infty}\frac{1}{n}h(u(n,\omega)x_{0})
\]
exist a.e. for any metric functional $h$. Evidence and discussion
of this appear in some of my earlier papers, for example \cite{K04}
(this reference also contains an argument why Theorem \ref{thm:ergodic}
in the special case of CAT(0)-spaces is equivalent to geodesic ray
approximation \cite{Ka00,KM99}). From ray approximation (being of
sublinear distance to a geodesic ray) and purely geometric reasons
(any two geodesic rays have a well-defined linear rate of asymptotic
divergence) all the above limits exist for proper CAT(0)-spaces and
Gromov hyperbolic spaces. See also \cite{Sa21} for a recent contribution
to this topic. 

\section{Multiplicative ergodic theorems for linear operators}

The need for multiplicative ergodic theorems for operators in infinite
dimensions has been expressed in the influential articles \cite{Ru82,ER85,LY12}.
In one approach to the 2D Navier-Stokes equation and related evolution
equations, the dynamics takes place in infinite dimensional Hilbert
spaces. There has been an increasing interest in results on this topics,
for example \cite{LL10,GTQ15,Bl16,MN20,BHL20}. González-Tokman wrote
in \cite{GT18} that ``An important motivation behind the recent
work on multiplicative ergodic theorems is the desire to develop a
mathematical theory which is useful for the study of global transport
properties of real world dynamical systems, such as oceanic and atmospheric
flows. Global features of the ocean flow include large scale structures
which are important for the global climate.'' In a different direction,
also leading to multiplicative ergodic theorems in Banach spaces is
\cite{CDS09}, which deals with difference equations with random delays.
Such delays are common in models of biological systems, immune response,
epidemiology, and economics (see \cite{CDS09} for references). 

Given a bounded linear operator $A$, submultiplicativity implies
that 
\[
\lim_{n\rightarrow\infty}\left\Vert A^{n}\right\Vert ^{1/n}
\]

exists and equals the spectral radius. On the other hand expressions
\[
\left\Vert A^{n}v\right\Vert ^{1/n}
\]
as $n\rightarrow\infty$ may not converge in infinite dimensions,
as is well known, see for example the introduction of \cite{Sc06}.
This puts a limitation on the validity of Oseledets theroem for operators.
Kingman's theorem takes care of the regularity of the growth of the
norm, and Theorem \ref{thm:ergodic} may be the appropriate replacement
for the second type of more directional behavior (local spectral theory).

The first infinite dimensional extension of Theorem \ref{thm:Osel}
is Ruelle's theorem \cite{Ru82} for compact operators, and other
early results were shown by Mañe and Thieullen \cite{Th87}. A strengthening
of this for the Hilbert-Schmidt class was obtained in \cite{KM99}: 
\begin{thm}
\emph{(\cite{KM99})} Let $u(n,\omega)$ be an ergodic cocycle of
$Id+A$ operators where $A$ is Hilbert-Schmidt. Then there is a.s.
an operator $\Lambda_{\omega}$ such that 
\[
\frac{1}{n}\left(\sum_{i}(\log\mu_{i}(n))^{2}\right)^{1/2}\rightarrow0
\]
as $n\rightarrow\infty$, where $\mu_{i}(n)$ are the eigenvalues
of the positive part of $\Lambda_{\omega}^{-n}u(n,\omega)$. 
\end{thm}

The uniformity of the convergence implicit in the conclusion, thanks
to the metric methods, is noteworthy since this is a much stronger
statement in infinite dimensions. In finite dimensions the statement
is equivalent to Oseledets' theorem. This metric approach was recently
substantially extended to a von Neumann algebra setting with a finite
trace in \cite{BHL20}, which used the theorem in \cite{KM99} together
with an intricate analysis, especially of completeness properties,
of a space of positive operators admitting a finite trace to get nonpositive
curvature. 

Let $Pos$ be the space of positive operators on a Hilbert space $H$.
This is a convex cone in the Banach space of symmetric operators,
and it has the corresponding Thompson metric: 
\[
d(p,q)=\sup_{v\in H,\left\Vert v\right\Vert =1}\left|\log\frac{(qv,v)}{(pv,v)}\right|.
\]

Invertible bounded operators $g$ act by isometry on this metric space
via $p\mapsto gpg^{*}$. Unless one restricts to subspaces where there
is a finite trace, this is not a CAT(0) space. On the other hand,
as noted in particular in \cite{CPR94}, the fundamental Segal inequality
\[
\left\Vert \exp(u+v)\right\Vert \leq\left\Vert \exp(u/2)\exp(v)\exp(u/2)\right\Vert 
\]
for symmetric operators $u$ and $v$, can be seen as a weak form
of nonpositive curvature more in Busemann's sense (with respect to
a selection of geodesics). This means that the exponential map $\textrm{exp:Sym\ensuremath{\rightarrow}Pos }$
is distance preserving on lines from $0$ and otherwise distance-increasing. 

In finite dimensions, Lemmens has recently determined $\overline{Pos}$
\cite{L21}. In infinite dimensions, the task to describe this compactification
remains to be done, with some small steps done in \cite{K22} in relation
to the invariant subspace problem.

As for ergodic cocycles of invertible bounded linear operators, one
has from Theorem \ref{thm:ergodic}:
\begin{thm}
\emph{(\cite{GK20}) }Let $v(n,\omega)=A(t^{n-1}\omega)...A(T\omega)A(\omega)$
be an integrable ergodic cocycle of bounded invertible linear operators
of a Hilbert space. Denote the square of the positive part 
\[
[v(n,\omega)]:=v(n,\omega)^{*}v(n,\omega).
\]
Then for a.e. $\omega$ there is a metric functional $h_{\omega}$
on Pos such that
\[
\lim_{n\rightarrow\infty}-\frac{1}{n}h_{\omega}([v(n,\omega)])=\lim_{n\rightarrow\infty}\frac{1}{n}\left\Vert \log[v(n,\omega)]\right\Vert .
\]
\end{thm}

Note that this statement by-passes the limitation mentioned above
from local spectral theory. We can deduce an a priori weaker statement
as follows, with a bit more information than in \cite{GK20}. A state
is a positive linear functional of norm 1 on certain types of algebras
of operators. The space of states is compact in the weak{*}-topology.
\begin{thm}
\emph{(\cite{GK20})} Let $v(n,\omega)=A(t^{n-1}\omega)...A(T\omega)A(\omega)$
be an integrable ergodic cocycle of bounded invertible linear operators
of a Hilbert space. Denote the square of the positive part 
\[
[v(n,\omega)]:=v(n,\omega)^{*}v(n,\omega).
\]
Then for a.e. $\omega$ there is a state $f_{\omega}$ on the space
of bounded linear operators of the form, 
\[
f_{\omega}(A)=s(A\xi,\xi)+(1-s)\psi(A),
\]
where $\xi$ is a unit vector, $\psi$ is a state on the algebra of
all bounded linear operators vanishing on all compact operators and
$0\leq s\leq1$, all depending on $\omega$, such that
\[
\lim_{n\rightarrow\infty}\frac{1}{n}\left|f_{\omega}(\log[v(n,\omega)])\right|=\lim_{n\rightarrow\infty}\frac{1}{n}\left\Vert \log[v(n,\omega)]\right\Vert .
\]
\end{thm}

\begin{proof}
Let $u(n,\omega):=v(n,\omega)^{*}$ and hence $[v(n,\omega)]=u(n,\omega)I$
is the random orbit in $\textrm{Pos }$. Therefore, as explained in
a previous section, we know that the distance 
\[
d(I,[v(n,\omega)])
\]
is a subadditive cocycle. Moreover, thanks to the exact distance properties
of the exponential map recalled above, also the distance inside $\textrm{Sym }$,
which is the operator norm 
\[
\left\Vert \log[v(n,\omega)]\right\Vert 
\]
is subadditive. To see this in detail: by the distance preserving
of $\exp$ on lines we have 
\[
d(I,[v(n,\omega)])=\left\Vert \log[v(n,\omega)]\right\Vert 
\]
while 
\[
\left\Vert \log[v(n,\omega)]-\log[v(n+m,\omega)]\right\Vert \leq d([v(n,\omega)],[v(n+m,\omega)])=d(I,[v(m,T^{n}\omega)])
\]
\[
=\left\Vert \log[v(m,T^{n}\omega)]\right\Vert .
\]
Thus by the triangle inequality, 
\[
\left\Vert \log[v(n+m,\omega)]\right\Vert \leq\left\Vert \log[v(n,\omega)]\right\Vert +\left\Vert \log[v(m,T^{n}\omega)]\right\Vert .
\]

Let $y_{n}:=\log[v(n,\omega)]$ and $\epsilon_{n}\searrow0$ . Since
$y_{n}$ is a self-adjoint operator we can find a unit vector $\xi_{n}$
such that
\[
\left|(y_{n}\xi_{n},\xi_{n})\right|>\left\Vert y_{n}\right\Vert -\epsilon_{n}.
\]
Let $f_{n}(A)=(A\xi_{n},\xi_{n})$ be the corresponding linear functional,
which in other words is a vector state. We assume that the right hand
side is strictly positive, that is,
\[
\tau:=\lim_{n\rightarrow\infty}\frac{1}{n}\left\Vert \log[v(n,\omega)]\right\Vert >0,
\]
otherwise there is nothing to prove. From the main subadditive ergodic
result in \cite{GK20} we have for almost every $\omega$, a sequence
$n_{i}\rightarrow\infty$ and sequence $\delta_{l}\rightarrow0$ such
that for every $i$ and every $l\leq n_{i}$,
\[
\left\Vert \log[v(n_{i},\omega)]\right\Vert -\left\Vert \log[v(n_{i}-l,T^{l}\omega)]\right\Vert \geq(\tau-\delta_{l})l.
\]

In view of this, for any $l\leq n_{i}$,
\[
\left\Vert y_{l}\right\Vert \geq\left|f_{n_{i}}(y_{l})\right|=\left|f_{n_{i}}(y_{n_{i}}+y_{l}-y_{n_{i}})\right|=\left|f_{n_{i}}(y_{n_{i}})-f_{n_{i}}(y_{n_{i}}-y_{l})\right|\geq\left|f_{n_{i}}(y_{n_{i}})\right|-\left|f_{n_{i}}(y_{n_{i}}-y_{l})\right|
\]
\[
\geq\left\Vert y_{n_{i}}\right\Vert -\epsilon_{n_{i}}-\left\Vert y_{n_{i}}-y_{l}\right\Vert \geq\left\Vert \log[v(n_{i},\omega)]\right\Vert -\left\Vert \log[v(n_{i}-l,T^{l}\omega)]\right\Vert -\epsilon_{n_{i}}
\]
\[
\geq(\tau-\delta_{l})l-\epsilon_{n_{i}}.
\]
By weak{*}-compactness letting $i\rightarrow\infty$ there is a state
$f=f^{\omega}$ for which
\[
\lim_{l\rightarrow\infty}\frac{1}{n}\left|f(y_{l})\right|=\tau,
\]
as desired. By Glimm's theorem \cite{Gl60} this state, being a limit
of vector states, must be of the form 
\[
f(A)=s(A\xi,\xi)+(1-s)\psi(A),
\]
where $\xi$ is a unit vector, $\psi$ a state on the algebra of all
bounded linear operators on the Hilbert space vanishing on all compact
operators and $0\leq s\leq1$.
\end{proof}
\begin{rem*}
When the cocycle is composed of compact operators then $s$ must be
$1$ and $f_{\omega}$ is a pure vector state, which then provides
a result pointing in the direction of Ruelle's theorem.
\end{rem*}

\section{Diffeomorphisms\label{sec:Diffeomorphisms}}

In the introduction topological Lyapunov exponents for surface homeomorphisms
were explained. Thurston's powerful measured foliation theory is presumably
difficult to generalize to dimensions greater than two (and so far
cannot treat the case of products of random homeomorphisms). The metric
perspective can on the other hand more easily be generalized. For
example, Ebin's Riemannian manifold of Riemannian metrics on a compact
manifold is one possibility \cite{Eb68}. The diffeomorphsisms act
by isometry. This is the replacement for the Teichmüller spaces. There
are two variants, one general and one restricted to Riemannian metrics
sharing the same volume form, and we would then consider volume preserving
diffeomorphisms. These spaces have nonpositive curvature but not always
complete.

Here is a related metric taken from \cite{AK21}, and which perhaps
has not been considered before. Let $M$ be a compact submanifold
of a finite dimensional vector space equipped with a norm $\left\Vert \cdot\right\Vert $.
Consider the following weak metric (which can be symmetrized if needed)
on the set $M$ of distance function functions bi-Lipschitz equivalent
to $d_{0}(x,y)=\left\Vert x-y\right\Vert $:
\[
D(d_{1},d_{2})=\log\sup_{x\neq y}\frac{d_{2}(x,y)}{d_{1}(x,y)}.
\]

If $T:M\rightarrow M$ is a diffeomorphism, it will preserve $D$-distances,
considered as map $(T^{*}d)(x,y):=d(Tx,Ty)$ since it just permutes
the underlying set $M$. Note that $T^{*}$ is an adjoint type of
map, it reverses the order of composition (and if it is desired to
keep the orientation we could instead use the inverse if it exists).

The following was shown in \cite{AK21}, with one ingredient being
the main results of \cite{GK20}: 
\begin{thm}
\emph{(Existence of a point with maximal stretch \cite{AK21}) }Let
\[
v(n,\omega)=A(t^{n-1}\omega)...A(T\omega)A(\omega)
\]
 be an integrable ergodic cocycle of diffeomorphisms of $M$. Then
there is a number $\lambda$ such that 
\[
\lim_{n\rightarrow\infty}\left(\sup_{x\neq y}\frac{\left\Vert v(n,\omega)x-v(n,\omega)y\right\Vert }{\left\Vert x-y\right\Vert }\right)^{1/n}=e^{\lambda}.
\]
In the case that $\lambda>0$ then there exists a point $z\in M$
and a sequence $w_{i}=(x_{i},y_{i})\in\left\{ (x,y)\in M\times M:x\neq y\right\} $
such that $w_{i}\rightarrow(z,z)$ and for any $\epsilon>0$ there
is $N>0$ such that for all $n>N$ 
\[
\frac{\left\Vert v(n,\omega)x_{i}-v(n,\omega)y_{i}\right\Vert }{\left\Vert x_{i}-y_{i}\right\Vert }\geq e^{(\lambda-\epsilon)n}
\]
for all $i$ sufficiently large for a fixed $n$. 
\end{thm}

In words, it means that given a cocycle there is a.e. a random point
$z$ such that nearby this point the cocycle is stretching at a near
maximal rate.

Another possibility of measuring distances in 1-dimensional dynamics
is the total variation of the logarithm of the derivative as in \cite{EBN22}.
This can be extended to a weak metric on diffeomorphisms groups on
compact manifolds $M$ via 
\[
d(f,g)=\sup_{x\in M}\left|\log\frac{\left|J_{g}(x)\right|}{\left|J_{f}(x)\right|}\right|
\]
where $\left|J_{f}\right|$ is the determinant of the Jacobian derivative.
See \cite{AK21} for further definitions of this type and their use.

\section{Applications of metrics and ergodic theorems in deep learning}

In \cite{AK21} Avelin and I introduced new geometric frameworks to
the theory of neural networks, that also enabled the application of
the noncommutative ergodic theorem. More specifically, we suggested
several metrics on the data set making various choices of layer maps
nonexpansive. This includes the most standard choices of activiation
functions (those mentioned above), and with positive, unitary or invertible
features for the weights. Since the composition of nonexpansive maps
remains nonexpansive, already this guarantees some regularity with
no wild fluctuations. 
\begin{itemize}
\item \emph{In several standard models of neural networks it is possible
to find semi-invariant metrics. This may help explaining certain phenomena
that often has been observed emprically, such as a certain stability
that ensures good generalization as opposed to the problem with the
overfitting phenomenon. }
\end{itemize}
In addition, in view of the noncommutative ergodic theorem above,
when the layer maps are selected at random and and the number of layers
is large, their compositions are close to being constant functions
in some cases, as in Theorem \ref{thm:Resnets} below. As Dherin and
his colleagues from Google and DeepMind informed us this fits very
well their theory aimed at explaining why deep learning generalize
well and do not overfit data, instead the trained network has a bias
towards simple functions \cite{DMRB22}. They measure simplicity with
their Geometric Complexity notion inspired by the Dirichlet energy.
Constant maps have zero complexity. And as the authors argue in \cite{BD21}
and \cite{DMRB22} initializations that deliver near-constant functions
is an advantage. 
\begin{itemize}
\item \emph{When the noncommutative ergodic theorem, Theorem \ref{thm:ergodic},
is applicable to the neural network, random initilialization gives
near-constant maps. According to \cite{BD21,DMRB22} this is something
observed in practice that may be highly desirable by contributing
to the harmonic nature of the functions that the stochastic gradient
method tends to find. }
\end{itemize}
Thus one can say that from this point of view it seems good to choose
a network architecture for which Theorem \ref{thm:ergodic} applies.
Such situations are discussed in more detail in \cite{AK21}. 

To illustrate the above points, here is a sample corollary of Theorem
\ref{thm:ergodic}: maps used in a popular layer model called ResNets,
see \cite{HZRS16}, are treated in the following result:
\begin{thm}
\emph{\label{thm:Resnets}(\cite{AK21})} Let $X=\mathbb{R}^{d}$
with the standard scalar product. Consider the layer maps $T(x)=W^{T}\sigma(Wx+b),$
with $b$ a general vector and $W$ having operator norm at most $1$,
and the activation function being either ReLU, TanH or the sigmoid
function. When such layer maps are selected i.i.d. under a finite
moment condition, it holds a.s. as $n\rightarrow\infty$ that there
is a (random) vector $v$ such that
\[
\frac{1}{n}T_{1}T_{2}...T_{n}x_{0}\rightarrow v.
\]
\end{thm}

The vector $v$ does not depend on the input $x_{0}$ thus the limiting
map is a (random) constant function, and in the case of a large, but
finite, fixed number $n$ of layers, the composed function should
be nearly constant.

The infinite-width limit as been more studied than the infinite-depth
case here discussed. For investigations about the case when both the
width and he depth go to infinity, see \cite{HaN20,LNR22}.

\emph{}

\end{document}